\documentclass[11pt,english,final]{amsart} % add final option to supress showkeys

\usepackage[margin=2.25cm]{geometry}
\usepackage{enumerate,enumitem}
\usepackage{graphicx,amsmath,amsthm,amssymb,mathrsfs}
\usepackage[dvipsnames]{xcolor}
\usepackage{twemojis}
\usepackage{setspace}
\usepackage[square,sort,comma,numbers]{natbib} 

\usepackage[colorlinks,linkcolor=black,citecolor=black]{hyperref} 
\usepackage{doi}
\usepackage[bordercolor=orange,backgroundcolor=orange!20,linecolor=orange,textsize=scriptsize,obeyFinal]{todonotes}
\usepackage[notcite,notref]{showkeys}  % remove color option if conflicts with tikz

\newtheorem{theorem}{Theorem}%[section]
\newtheorem{definition}[theorem]{Definition}

\newtheorem{lemma}[theorem]{Lemma}

\newtheorem{setup}[theorem]{Setup}

\setlist{itemsep=2pt,parsep=1pt,topsep=2pt,partopsep=0pt}  
\setenumerate{leftmargin=*,labelindent=\parindent} 

\def\itm#1{\rm ({#1})} 
\def\itmit#1{\itm{\it #1\,}} 
 
\def\abc{\itmit{\alph{*}}}

\def\step1#1{\itm{Step {\it #1\,}}}

\newcommand{\By}[2]{\overset{\mbox{\tiny{#1}}}{#2}}
\newcommand{\ByRef}[2]{   \By{\eqref{#1}}{#2} }

\newcommand{\leByRef}[1]{ \ByRef{#1}{\le} }

\renewcommand{\subset}{\subseteq}
\newcommand{\im}{\mathrm{Im\,}}
\newcommand{\dcup}{\dot\cup}
\newcommand{\eps}{\ensuremath{\varepsilon}} 
\renewcommand{\rho}{\varrho}
\renewcommand{\phi}{\varphi}
\newcommand{\Prob}{\mathbb{P}}
\newcommand{\Exp}{\mathbb{E}}

\newcommand{\cB}{\ensuremath{\mathcal{B}}}

\newcommand{\cE}{\ensuremath{\mathcal{E}}}
\newcommand{\cF}{\ensuremath{\mathcal{F}}}
\newcommand{\cX}{\ensuremath{\mathcal{X}}}
\newcommand{\cY}{\ensuremath{\mathcal{Y}}}

\DeclareMathOperator{\dom}{dom}

\author[P. Allen]{Peter Allen}
\address{(PA,JB,JS) London School of Economics, Department of Mathematics, Houghton
  Street, London WC2A 2AE, UK}
\email{p.d.allen|j.boettcher|j.skokan@lse.ac.uk}
\author[J. B\"ottcher]{Julia B\"ottcher}
\author[J. Skokan]{Jozef Skokan}
\author[B. Sudakov]{Benny Sudakov}
\address{(BS) Department of Mathematics, ETH, Ramistrasse 101, 8092 Zurich, Switzerland}
\email{benjamin.sudakov@math.ethz.ch}
\thanks{(BS) Research supported in part by SNSF grant 200021-19696.}

\title{Breaking the Bollob\'as--Eldridge--Catlin Barrier for Bipartite Graphs}
\date{}

\begin{document}

\begin{abstract}
The celebrated Bollob\'as--Eldridge--Catlin packing conjecture states
that every $n$-vertex graph $G$ with minimum degree at least
$\big(1-\frac{1}{\Delta+1}\big) n$ contains every $n$-vertex graph $H$ of
maximum degree at most $\Delta$. Despite considerable attention, the
conjecture remains widely open.

We show that for bipartite $H$ this threshold can be greatly improved:
there is an absolute constant $c>0$ such that every $n$-vertex graph
$G$ with minimum degree at least
$
\big(1-c\frac{\log\Delta}{\Delta}\big)n
$
contains every $n$-vertex bipartite graph $H$ of maximum degree at most
$\Delta$, provided $\Delta$ is not too large compared to $n$. Moreover,
we prove that this logarithmic improvement is best possible up to the
value of the constant.
\end{abstract}

\maketitle

\thispagestyle{empty}
\section{Introduction}
A central theme in extremal combinatorics is to determine how large the
minimum degree of an $n$-vertex graph $G$ must be in order to contain a given $n$-vertex graph $H$. A classical early example
is Dirac's theorem~\cite{Dirac}, which says that, for $n\ge 3$,
minimum degree $n/2$ guarantees a Hamilton cycle. Another celebrated
result of this kind is the Hajnal--Szemer\'edi theorem~\cite{HS}, which
says that, for every $\Delta\ge2$, minimum degree
$\frac{\Delta}{\Delta+1}n$ guarantees a spanning $K_{\Delta+1}$-factor
whenever $\Delta+1$ divides $n$.

A fundamental conjecture of Bollob\'as and Eldridge~\cite{BE} and,
independently, Catlin~\cite{C1,C2} from 1970's,
generalises the Hajnal--Szemer\'edi
theorem from clique-factors to arbitrary bounded-degree spanning
graphs. 
It states that any two $n$-vertex graphs $G_1$ and $G_2$ pack, that is,
they admit edge-disjoint copies on the same vertex set, whenever
$
(\Delta(G_1)+1)(\Delta(G_2)+1)\le n+1.$
Taking $G_1=H$ and $G_2=\overline G$, this implies the following
minimum-degree formulation. Every $n$-vertex graph $G$ with minimum
degree at least $\frac{\Delta}{\Delta+1}n$ should contain every
$n$-vertex graph $H$ with maximum degree at most $\Delta$. Despite a lot of attention, the
conjecture remains open, though it was proved for $\Delta=2$ by Aigner
and Brandt~\cite{AB}, for $\Delta=3$ and for sufficiently large $n$, by Csaba, Shokoufandeh and Szemer\'edi~\cite{CSS}, and for $\Delta=4$ and sufficiently large $n$, by Csaba~\cite{csabaphd}. Somewhat weaker
results are known in general. Sauer and Spencer~\cite{SauSpe}
proved that a minimum degree greater than
$\big(1-\frac{1}{2\Delta}\big)n$ suffices for all $n$-vertex graphs of
maximum degree at most $\Delta$, and Kostochka, Kaul and Yu~\cite{KKY}
improved this to $\big(1-\frac{3}{5(\Delta+1)}\big)n$ for
$\Delta\ge300$. An approximate form of the conjecture was established by Eaton~\cite{eaton}. She proved that if  $(\Delta(G_1)+1)(\Delta(G_2)+1)\le n+1$, then the graphs $G_1, G_2$ admit a near packing in which the graph formed by the conflicting edges has maximum degree at most one. 

Further results exploit additional sparsity of one of the graphs. For example, Bollob\'as, Kostochka and Nakprasit~\cite{BKNdegenerate} obtained stronger packing conditions when one graph has small degeneracy. Kaul and Reiniger~\cite{KaulReiniger} later gave a degree-sequence refinement of the Sauer--Spencer theorem for degenerate graphs. Finally, Cames van Batenburg and Kang~\cite{CamesKang} proved the Bollob\'as--Eldridge--Catlin conjecture in an asymmetric regime when one graph has bounded codegree.

In this paper, we consider the bipartite case of this problem, that is, we restrict our attention to the containment of bipartite $n$-vertex
graphs $H$ with $\Delta(H)\le\Delta$. This setting was first considered
by Csaba~\cite{bela07}, who proved that the
Bollob\'as--Eldridge--Catlin conjecture holds for bipartite graphs. In
fact, he showed the stronger statement that minimum degree
$(1-\beta)\frac{\Delta}{\Delta+1}n$ already suffices to contain every
bipartite $n$-vertex graph $H$ with maximum degree at most $\Delta$.
However, Csaba's proof gives only a very small value of $\beta=\beta(\Delta)$,
with $\beta(\Delta)\to0$ as $\Delta\to\infty$. Moreover, it relies on
the Szemer\'edi Regularity Lemma, and therefore applies only when
$\Delta$ grows very slowly with $n$, roughly like $\log_* n$.

The aim of this paper is to determine the order of magnitude of the
best possible improvement over the Bollob\'as--Eldridge--Catlin
threshold for bipartite graphs. Our main theorems resolve this question, giving matching upper and lower bounds as follows.

\begin{theorem}[upper bound]\label{thm:upper-bound}
  For all sufficiently large $n$ and each $2\cdot e^{1000}\le\Delta\le 0.001{n^{1/5}}$, the following holds. Every $n$-vertex graph~$G$ with minimum degree 
  \[\delta(G)\ge\Big(1-\frac{\log\Delta}{500\Delta}\Big)n\,,\]
  contains each
  $n$-vertex bipartite graph~$H$ with maximum degree at most~$\Delta$.
\end{theorem}

\begin{theorem}[lower bound]\label{thm:lower-bound}
  For each $\Delta\ge 2$ and each $n\ge\Delta+1$, there is an $n$-vertex bipartite graph~$H$ with maximum degree at most~$\Delta$ and an $n$-vertex graph~$G$ with minimum degree \[\delta(G)\ge\Big(1-164\frac{\log\Delta}{\Delta}\Big)n\] such that~$G$ does not contain~$H$.
\end{theorem}

\noindent
We made no attempt to optimise the constants in these theorems, since 
the point of these results is to determine the correct order of magnitude of
the improvement over the Bollob\'as--Eldridge--Catlin threshold for
bipartite graphs.

\subsection*{Outline of the proofs}
Let us briefly describe the ideas behind the proofs. For the lower bound, we construct a bipartite graph $H$ with no
bipartite holes of size $\Theta(\frac{\log\Delta}{\Delta}n)$, that is, no pair of disjoint sets of this size without any edges between them. As host graph~$G$ we take a $\Theta(\frac{\log\Delta}{\Delta}n)$-blow-up of the complement of a graph with constant maximum degree and the following mild expansion property: there is an edge between any pair of disjoint sets with at least a quarter of the vertices. We call the blow-ups of single vertices in this construction a cluster of~$G$. We then show that any attempted embedding of $H$ into $G$ forces many vertices from the two partition classes of $H$ into two clusters of $G$ with no edges between them.
This creates a large bipartite
hole in $H$, a contradiction. This shows that the logarithmic gain over the Bollob\'as-Eldridge-Catlin threshold in our upper bound is best possible. 

The upper bound is
proved by a random embedding argument, rather than by the regularity
lemma. Given a bipartite graph $H$ with parts $X$ and $Y$, we first
randomly split the vertex set of $G$ into two parts $U$ and $V$ of the
corresponding sizes (see Lemma~\ref{lem:makebip}). Our aim will now be to embed $H$ into the bipartite subgraph $G[U,V]$ of $G$, which with high probability inherits the high minimum degree of $G$. We give a randomised embedding of $X$ into $U$ and with high probability can complete this embedding by embedding $Y$ into $V$. This completion step is equivalent to verifying Hall's condition in an auxiliary bipartite graph with parts $Y$ and $V$ whose edges encode common neighbourhoods in $G$, so that the difficulty is to find a good randomised embedding of $X$ into $U$ (we will show that Lemma~\ref{lem:stage3} implies that Hall's condition holds in this auxiliary graph).

The first idea is simply to pick a uniform random embedding of $X$ into $U$, and if $\Delta$ is large enough (at least polylogarithmic in $n$) this turns out to suffice. If $\Delta$ is smaller, the uniform random embedding almost works, but a few errors might be made, corresponding to vertices of $Y$ with small or zero degree in the auxiliary bipartite graph. We fix these errors by a random re-embedding of the problematic vertices, using a version of the Lov\'asz local lemma and a reservoir of carefully chosen set-aside vertices of $X$ (this is captured in Lemma~\ref{lem:stage2}).

A main probabilistic
ingredient to our proof is that, in a bipartite graph with high minimum degree, a uniform random $\Delta$-set is very likely to have a large common neighbourhood. We show, using a martingale concentration argument, that  this holds with a failure probability that is exponentially small in $\Delta$ (see Lemma~\ref{lem:main}), which is sharp.

Since our proof
does not use regularity, it applies when $\Delta$ grows much faster with
$n$ than in Csaba's theorem. The restriction $\Delta\le 0.001{n^{1/5}}$
for our upper bound is not intrinsic to the problem: our method can
push this range somewhat further, at the cost of worse constants. On the
other hand, some restriction on the growth of $\Delta$ is necessary, as random graphs satisfying the required minimum degree condition
need not contain even a copy of $K_{\Delta,\Delta}$ when $\Delta$ is too
large.

\subsection*{Organisation} In Section~\ref{sec:tools} we introduce the notation we shall use in the remainder of the paper as well as some probabilistic tools. In Section~\ref{sec:lower} we prove Theorem~\ref{thm:lower-bound}. In Section~\ref{sec:upper}
we prove Theorem~\ref{thm:upper-bound} modulo our main technical lemma (Lemma~\ref{lem:main}) and some lemmas analysing the different stages of our embedding algorithm (Lemmas~\ref{lem:stage3}--\ref{lem:stage2}). 
The former is then proved in Section~\ref{sec:lemma}, and the latter in Section~\ref{sec:stages}. Section~\ref{sec:concl} closes with some concluding remarks. Throughout the paper, we will often omit floor and ceiling signs whenever they are not crucial, in order to improve readability.

\section{Notation and tools}\label{sec:tools}

Unless stated otherwise, our logarithms use base~$e$.
For a graph~$G$, we write $V(G)$ and $E(G)$ for the vertex set and the edges set of~$G$, respectively. For disjoint set $U,V\subset V(G)$ and a vertex $u\in V(G)$ we write $N_G(u,V)$ for the neighbourhood of~$u$ in~$V$ and 
\[\delta_G(U,V)=\min\big\{|N_G(u,V)|\colon u\in U\big\}\]
for the minimum degree to~$V$ among vertices in~$U$. Further, we write $N_G(U)$ for the \emph{common neighbourhood} of~$U$. We shall use the following logarithm estimate:
\begin{equation}\label{eq:log}
0\le -\log(1-x)\le 2x \qquad
    \text{if} \quad 
    0\le x\le\frac12 
\end{equation}

We will use the following form of Chernoff's inequality for binomially and hypergeometrically distributed random variables.
Recall that the \emph{hypergeometric distribution} with parameters $(n,\ell,h)$ is the number of elements from $[\ell]:=\{1,\dots,\ell\}$ selected when a subset of $[n]$ of size $h$ is picked uniformly at random.

\begin{lemma}[Corollary~2.3 and Theorem~2.10 of~\cite{purple-book}]\label{lem:hypgeo}
    Given positive integers $n,\ell,h$, if $X$ is hypergeometrically distributed with parameters $(n,\ell,h)$ then $\Exp[X]=\frac{\ell h}{n}$. If $X$ is either hypergeometrically or binomially distributed, then for each $\xi\in(0,\tfrac32)$ we have
    \[\Prob\big[|X-\Exp[X]|>\xi\Exp[X]\big]<2\exp\big(-\tfrac13\xi^2\Exp[X]\big)\,.\]
\end{lemma}

In addition, we will apply the following martingale-type concentration bound, which is a consequence of Freedman's inequality.
Let $\Omega$
be a finite probability space. A \emph{filtration} $(\cF_0$,
$\cF_1$, \dots, $\cF_n)$ is\footnote{This is not the standard definition of a filtration, which would replace each $\cF_i$ with the $\sigma$-algebra it generates; this version is more convenient here.} a sequence of partitions of~$\Omega$
such that $\cF_i$ refines $\cF_{i-1}$ for all $i\in[n]$.  For a filtration $(\cF_0$, $\cF_1$, \dots, $\cF_n)$ we say
that a function $f\colon\Omega\rightarrow \mathbb{R}$ is
\emph{$\cF_i$-measurable} if $f$ is constant on each part of
$\cF_i$. Further, for any random variable
$Y\colon\Omega\rightarrow\mathbb{R}$ the \emph{conditional expectation}
$\Exp(Y|\cF_i)\colon\Omega\rightarrow\mathbb{R}$ with respect to $\cF_i$ is defined by
$\Exp(Y|\cF_i)(x)=\Exp(Y|X)$, where $X\in\cF_i$ is such that $x\in X$.
The following concentration bound follows for example from \cite[Corollary~7]{ABCT:leaves}.

\begin{theorem}
\label{thm:freedman}
  Let $\Omega$ be a finite probability space and $(\mathcal{F}_0,
  \mathcal{F}_{1},\dots,\mathcal{F}_{n})$ be a filtration. Let $\mu,R>0$, and let $Y_{i}$ be an
  $\mathcal{F}_i$-measurable non-negative random variable for each $i\in[n]$.
  If $\sum_{i=1}^{n}\Exp\big[Y_{i}\big|\mathcal{F}_{i-1}\big]\le\mu$ and $0\le Y_i\le R$ for each $i\in[n]$, then
  \begin{equation*}
      \Prob\bigg[\sum_{i=1}^{n}Y_{i} \geq 2\mu\bigg] \le\exp\Big(-\frac{\mu}{4R}\Big)\,.
  \end{equation*}
\end{theorem}

Finally, we need the local lemma, for which the following definition is required.

\begin{definition}[$p$-dependency graph]
  Let~$\cX$ be a set of events in some probability space, and let $p\in[0,1]$. Then~$D=(\cX,E)$
  is a \emph{$p$-dependency graph} for~$\cX$ if every event~$X\in\cX$ satisfies the following condition: for every
  set of events~$\cY\subset\cX \setminus N_D(X)$ we have $\Prob(X|\bigcap_{Y\in\cY}\bar Y)\le p$.
\end{definition}

The local lemma states that, if a given set~$\cX$ of bad events has a
sufficiently sparse $p$-dependency graph compared to $p$, then with positive probability no bad event occurs.

\begin{lemma}[Lov\'asz local lemma~\cite{ErdLov75}]
\label{lem:LLLL}
  Given $p\in[0,1]$, let $\cX=\{X_i\}_{i\in[t]}$ be some set of events with $p$-dependency
  graph~$D=(\cX,E)$. If $4\Delta(D)p<1$, then $\Prob(\bigcap_{i\in[t]}\bar X_i)>0$.
\end{lemma}

We note that this statement is slightly different to that in~\cite{ErdLov75}, but their proof goes through almost verbatim. Specifically, they use `lopsidependency' to write $\Prob[X_m|\bar{X}_{d+1},\dots,\bar{X}_s]\le\Prob[X_m]\le p$, where $X_{d+1},\dots,X_s$ are non-neighbours of $X_m$ in $D$. Our definition of $p$-dependency graph gives $\Prob[X_m|\bar{X}_{d+1},\dots,\bar{X}_s]\le p$ directly, and this is all that is required in their proof.

\section{A lower bound construction}
\label{sec:lower}

In this section we prove Theorem~\ref{thm:lower-bound}.
For constructing the graphs~$G$ and~$H$ in this theorem, we shall use the following two lemmas providing graphs with suitable expansion 
properties.
In a bipartite graph~$H$ with bipartition $(X,Y)$ we call a pair $(X',Y')$ with $X'\subset X$ and $Y'\subset Y$ a \emph{bipartite hole} of size $\ell$ if $|X'|=|Y'|=\ell$ and there are no edges between $X'$ and $Y'$.

\begin{lemma}\label{lem:bi-expander}
  For any $n$, any $\Delta\ge 3$
  and any $0<\tau<1$ such that $\Delta>2(1-\log\tau)/\tau$, there is a bipartite graph $H$ with bipartition $(X,Y)$ and $|X|=|Y|=n$, whose maximum degree is at most $\Delta$,
  that has no bipartite holes of size at least $\tau n$.
\end{lemma}
\begin{proof}
    Consider a collection of $\Delta$ independent uniform random matchings $M_1$, $M_2, \dots,M_\Delta$ on $(X,Y)$. For fixed sets $A$ and $B$ of size $t=\lceil \tau n\rceil$ in $X$ and $Y$ respectively, the probability that for a fixed $i\in[\Delta]$ the matching $M_i$ does not match any vertex of $A$ to a vertex of $B$ is
    \[\frac{(n-t)\cdot(n-t-1)\cdot\ldots\cdot(n-2t+1)\cdot(n-t)!}{n!}\le\Big(\frac{n-t}{n}\Big)^t\,.\]
    Letting $H=\bigcup_{i=1}^\Delta M_i$,
    we see that the probability of $(A,B)$ being a bipartite hole in $H$ is at most
    \[
     \Big(\frac{n-t}{n}\Big)^{t\Delta}
     =
     \Big(1-\frac{t}{n}\Big)^{t\Delta}
     \le
     \exp\Big(-\frac{t^2\Delta}{n}\Big)\,.\]

    The number of choices of $(A,B)$ is $\binom{n}{t}^2\le\big(\tfrac{en}{t}\big)^{2t}$, so by the union bound the probability that $H$ contains a bipartite hole with parts of size $t$ is at most
    \[
     \Big(\frac{en}{t}\Big)^{2t}  
           \exp\Big(-\frac{t^2\Delta}{n}\Big)
     \le \exp\bigg(2t\Big(\log\Big(\frac{e}{\tau}\Big)-\frac{\tau\Delta}{2}\Big)\bigg)\,,
    \]
    which is smaller than $1$ when $\frac{\tau\Delta}{2}>1-\log\tau$.
\end{proof}

As a corollary of Lemma~\ref{lem:bi-expander} we obtain the following.

\begin{lemma}\label{lem:expander}
  For any $m$ there is an $m$-vertex graph $F$ with maximum degree $\Delta(F)\le40$ such that for any pair $A,B$ of disjoint vertex sets in $F$ with $|A|,|B|\ge m/4$ we have $e(A,B)>0$.
\end{lemma}
\begin{proof}
    Let $\Delta=20$ and observe that $\Delta>2(1-\log\frac14)/\frac14$.
    Hence, the graph $H$ of Lemma~\ref{lem:bi-expander} with $2m$ vertices and maximum degree $\Delta$ whose parts are $X=\{x_1,\dots,x_m\}$ and $Y=\{y_1,\dots,y_m\}$
    has no bipartite hole of size $\frac14 m$.
    
    Let $F$ be obtained from~$H$ by identifying pairs of vertices in $X$ and $Y$. More precisely, $F$ has vertex set $Z=\{z_1,\dots,z_m\}$ and $z_iz_j$ is an edge of~$F$ precisely when $x_iy_j\in E(H)$ or $y_ix_j\in E(H)$.
    Note that the maximum degree of $F$ is at most $2\Delta=40$. 
    
    Now, suppose that $(A,B)$ is a pair of disjoint vertex sets in~$F$, each of size $\frac14m$ and with no edges between them. Let $A'$ be the set of vertices in $X$ corresponding to $A$ and $B'$ the set of vertices in $Y$ corresponding to $B$. Then $(A',B')$ forms a bipartite hole in $H$, a contradiction.
\end{proof}

The graph~$G$ in our construction will be a blow-up of the complement $\bar F$ of the graph~$F$ given in this last lemma. Here, a \emph{$k$-blowup} of~$\bar F$ is the graph obtained from~$\bar F$ by replacing each of its vertices by an independent set of size~$k$ and each of its edges by a complete bipartite graph between the corresponding independent sets. We will call these independent sets replacing a vertex in such a blowup the \emph{clusters} of the blowup.

\begin{proof}[Proof of Theorem~\ref{thm:lower-bound}]
  Given $\Delta\ge 2$ and $n\ge\Delta+1$, let $\tau=2\frac{\log\Delta}{\Delta}$ and set $m=\frac{\Delta}{4\log\Delta}$. 
  Observe that if $\Delta\le 200$, then the minimum degree condition in the theorem is vacuous. We can choose $G$ to be the $n$-vertex graph with no edges and $H$ any $n$-vertex graph with at least one edge. Hence we can assume $\Delta\ge200$ from this point.
  
  Let~$H'$ be the graph from Lemma~\ref{lem:bi-expander} with its bipartition $(X,Y)$, each part of size $\lceil\tfrac{n}{2}\rceil$, and with maximum degree at most~$\Delta$, which has no bipartite hole of size at least $\tau\lceil\frac{n}{2}\rceil\le\frac{4}{3}\frac{\log\Delta}{\Delta}n=\frac{n}{3m}$.
  This graph~$H'$ exists because
  \[2\frac{1-\log\tau}{\tau}=\Delta\Big(\frac{1-\log 2-\log\log\Delta+\log\Delta}{\log\Delta}\Big)<\Delta\,,\]
  where the last inequality uses $\Delta\ge 100$.
  If~$n$ is even let $H=H'$, and if~$n$ is odd let~$H$ be obtained from~$H'$ by deleting an arbitrary vertex. Then~$H$ still has maximum degree at most~$\Delta$ and no bipartite hole of size at least $\frac{2n}{3m}$.
  
  Let $F$ be the graph on $m$ vertices given by Lemma~\ref{lem:expander} with maximum degree at most~$40$ and edges between every pair of disjoint vertex sets of size at least $m/4$.
  Consider the  complement~$\bar F$ of~$F$. Let~$G'$ be the $\lceil n/m\rceil$-blowup of $\bar F$, and let~$G$ be a graph on~$n$ vertices obtained from~$G'$ by deleting from each cluster at most one vertex.
  We easily check that for $n\ge 100$ we have
  \[\delta(G)\ge n-41\frac{n}{m}
  =\Big(1-164\frac{\log\Delta}{\Delta}\Big)n\,.\]
  It remains to show that~$H$ is not a subgraph of~$G$. Assume, for a contradiction, that we can embed~$H$ into~$G$ and fix such an embedding. We now label a cluster of~$G$ by~$X$ if at least half of its vertices host a vertex from~$X$ in this embedding, and by~$Y$ otherwise.
  Now, observe that if a cluster~$C$ is labelled~$X$ (respectively~$Y$), then in fact at least $\frac23|C|$ of its vertices are in~$X$ (respectively~$Y$). Otherwise, the vertices of $H$ embedded to~$C$ would contain a bipartite hole of size at least $\frac{n}{3m}$, contradicting the choice of~$H$.
  
  We conclude that at most $\frac{3}{4}m$ clusters of~$G$ are labelled~$X$, and similarly at most $\frac34m$ clusters are labelled $Y$. It follows that at least $\frac{1}{4}m$ clusters of~$G$ are labelled~$Y$ and at least $\frac{1}{4}m$ clusters of~$G$ are labelled~$X$.
  Hence, by out choice of $F$, there is an edge of $F$ from a cluster $A$ labelled $X$ to a cluster $B$ labelled $Y$, so $G$ has no edges between these clusters. Letting $A'$ be the vertices of $X$ embedded to $A$, and $B'$ the vertices of $Y$ embedded to $B$, we have that $(A',B')$ is a bipartite hole in $H$ of size at least $\frac{2n}{3m}$, contradicting the choice of~$H$.
\end{proof}

\section{Proof of the upper bound}
\label{sec:upper}

In this section we prove Theorem~\ref{thm:upper-bound}. The proof is based on the following natural probabilistic embedding argument: given $H$ with its bipartition classes $X$ and $Y$, we split $G$ randomly into two parts $U$ and $V$ of sizes $|X|$ and $|Y|$. We embed $X$ into $U$, and the embedding of $Y$ into $V$ is then equivalent to finding a perfect matching in an auxiliary bipartite graph $F$ with parts $Y$ and $V$, putting an edge from $y\in Y$ to $v\in V$ if $N_H(y)$ was embedded to $N_G(v)$.

To find this perfect matching, we check Hall's condition. Specifically, we will prove a minimum degree condition on $F$, and that any large enough subset of $V$ dominates most of $Y$. We will see that Hall's condition follows easily from these two properties.

The first idea for embedding $X$ into $U$ is simply to pick a uniform random injection, and if $\Delta$ is sufficiently large this is exactly what we will do. If $\Delta$ is small, however, this strategy does not quite work: we obtain the required minimum degree for vertices in $V$, and the large subsets condition, but a very tiny number of vertices in $Y$ might have few or no neighbours in $V$. 
In this case, we need to be a bit more careful. We set aside a small `reservoir' of far-apart vertices in $X$ and randomly embed the rest into $U$. We then re-embed the few vertices of $X$ which led to the minimum degree condition on $F$ failing, and finally embed the reservoir arbitrarily. We will show that this guarantees the required conditions on $F$.

As we shall explain below, the main lemma we need to analyse this embedding approach is the following, which states that in a bipartite graph with high minimum degree, very few sets of size at most~$\Delta$ have a small common neighbourhood. 

\begin{lemma}[main lemma]\label{lem:main}
    For all $0<\eps\le 1/3$ and $\Delta\ge2$ the following holds. Let $G'$ be a bipartite graph with parts $A'$ and $B'$, with $|A'|\ge2\Delta$ and $|B'|\ge 1$.
    Suppose that for all $a\in A'$ and all
    $b\in B'$ we have
    \begin{equation*}
      \deg(a)\ge \Big(1-\frac{\eps\log\Delta}{16\Delta}\Delta^{\eps}\Big)|B'|
      \quad\text{and}\quad
      \deg(b)\ge \Big(1-\frac{\eps\log\Delta}{16\Delta}\Big)|A'|\,.
    \end{equation*}
    For any given $1\le t\le\Delta$, let~$T\subseteq A'$ be a set chosen uniformly at random among all subsets of~$A'$ of size~$t$.
    Then
    \[\Prob\Big(\big|N_{G'}(T)\big|\le \Delta^{-\eps}|B'|\Big)\le \exp\big(-\Delta^{1-2\eps}\big)\,.\]
\end{lemma}

Let us now provide the details of our embedding strategy and the properties this strategy guarantees. We shall use the following setup, which in particular defines the reservoir set we shall use.

\begin{setup}\label{setup}
  Let
  \[c=\frac{1}{500}\,,\qquad 2\cdot e^{1000}\le\Delta\le 0.001n^{1/5}\,,\qquad\text{and}\qquad\eps=\frac13\,.\]
  Let~$H$ be a bipartite graph on~$n$ vertices with maximum degree $\Delta(H)\leq \Delta$ and with bipartition classes $V(H)=X\dcup Y$ with $|X|\ge|Y|\ge|X|/\Delta$. If $\Delta\le \big(\log(2n)\big)^{1/(1-2\eps)}$, we set aside a set $\tilde X\subset X$, which we call the \emph{reservoir}, of vertices of distance at least~$4$ and with $|\tilde X|=\lceil |X|/\Delta^{2+3\eps}\rceil$. If $\Delta> \big(\log(2n)\big)^{1/(1-2\eps)}$, we set $\tilde X=\emptyset$.
  
  Let~$G$ be a graph with minimum degree $\delta(G)\ge\big(1-c\log\Delta/\Delta\big)n.$
  Let $V(G)=U\dcup V$ be a partition of the vertex set of~$G$ with $|U|=|X|$ and $|V|=|Y|$ such that
  \begin{equation*}
    \delta_G(U,V)\ge\Big(1-2c\frac{\log\Delta}{\Delta}\Big)|V|\,,\qquad
    \delta_G(V,U)\ge\Big(1-2c\frac{\log\Delta}{\Delta}\Big)|U|\,. 
  \end{equation*}
\end{setup}

Observe that the existence of the set $\tilde X$ in this setup follows from the following greedy algorithm: Add an arbitrary vertex $x\in X$ to $\tilde X$, cross out the at most $\Delta(\Delta-1)$ vertices at distance~$2$ from~$x$, add an arbitrary remaining vertex of~$X$ to~$\tilde X$ and repeat. Since $H$ is bipartite and
$\tilde X\subset X$, this ensures that any two vertices of $\tilde X$ are at distance at least $4$ in $H$.
The existence of the partition $V(G)=U\dcup V$ follows from the following standard lemma.

\begin{lemma}\label{lem:makebip}
  Let $\rho>0$ and let 
  $a,n\in\mathbb{N}$ be such that $\rho a, \rho(n-a)\ge6\log n$.
  Let~$G$ be a graph on~$n$ vertices with minimum degree $\delta(G)\ge(1-\rho)n$. A random partition of $V(G)$ into parts $A$ and $B$ with $|A|=a$ satisfies $\delta_G(A,B)\ge\big(1-2\rho\big)|B|$ and $\delta_G(B,A)\ge\big(1-2\rho\big)|A|$
  with probability at least $1-4n^{-1}$.
\end{lemma}
\begin{proof}
    Given $G$, let $A$ be a uniform random subset of $V(G)$ of size $a$. Given $x\in V(G)$, let $S$ be a superset of $V(G)\setminus N(x)$ of size $\rho n$. By Lemma~\ref{lem:hypgeo}, with $\xi=1$, the probability of $|A\cap S|\ge 2\rho |A|$ is at most $2\exp\big(-\tfrac13\rho|A|\big)$.

    Taking the union bound over vertices~$x$ and using the lower bound on $a$, the probability that there exists a vertex of $G$ with more than $2\rho|A|$ non-neighbours in $A$ is at most $2n^{-2}\cdot n$. The same argument applies to $B$, which is  a uniform random subset of $V(G)$ of size $n-a$. Hence, either claimed degree condition fails with probability at most $4n^{-1}$.
\end{proof}

Starting from Setup~\ref{setup}, we construct the desired embedding of~$H$ into~$G$ using the following algorithm. We shall first give all the details of this algorithm and then formulate the properties that we claim hold after each stage. In these stages we construct various embeddings $\phi$ of vertices from $X$ to~$U$. For such an embedding with domain $\dom(\phi)$, we say that 
a vertex $y\in Y$ is \emph{$p$-problematic} for $\phi$ if the vertices of~$G$ hosting the neighbours of~$y$ embedded by~$\phi$ have a small common neighbourhood, that is, if
\[\Big|N_G\Big(\phi\big(N_H(y)\cap\dom(\phi)\big)\Big) \cap V\Big|\le\frac{|V|}{p}\,.\]

\textbf{Stage 1:} Let $\phi_1\colon X\setminus\tilde X\to U$ be an injection that is chosen uniformly at random.
Let 
\begin{align*}
Y_1&= \big\{y\in Y\colon \text{$y$ is $\Delta^\eps$-problematic for $\phi_1$}\big\}\,,\\
X'&= \big\{x\in X\setminus\tilde X\colon Y_1\cap N_H(x)\neq\emptyset\big\}\,, \qquad\qquad\\
X_1&=X\setminus(\tilde X\cup X')\,.
\end{align*}

\textbf{Stage 2:} We construct $\phi_2\colon X\setminus\tilde X\to U$ as follows. On $X_1$ we set $\phi_2=\phi_1$. On $X'$ we choose $\phi_2$ uniformly at random among all injections from $X'$ to $U\setminus\phi_1(X_1)$.
% , and with positive probability there are no $\Delta^{2\eps}$-problematic vertices for $\phi_2$.

\textbf{Stage 3:} We construct $\phi_3\colon X\to U$ by letting $\phi_3=\phi_2$ on $X\setminus\tilde X$, and embedding the vertices of $\tilde X$ arbitrarily into $U\setminus\im\phi_2$.

\textbf{Stage 4:} We construct the auxiliary bipartite graph~$F$ with vertex set $V(F)=Y\dcup V$ and with all edges $yv$ with $y\in Y$ and $v\in V$ such that $v\in N_G\big(\phi_3(N_H(y))\big)$. We find a perfect matching~$M$ in~$F$. The embedding $\phi_4\colon X\dcup Y\to U\dcup V$ of~$H$ into~$G$ is then given by letting $\phi_4=\phi_3$ on~$X$ and for each $y\in Y$ letting $\phi_4(y)=v$ such that $yv\in M$.

\medskip

If $\Delta$ is large, we will see that it is likely that after Stage 1, $Y_1$ is empty, so that $X'=\emptyset$ and Stages 2 and 3 become trivial with $\phi_3=\phi_2=\phi_1$. In general, as we will show, the crucial achievement of Stage~2 is that with positive probability there are no $\Delta^{2\eps}$-problematic vertices left after this stage.

Since $\phi_3$ embeds all of the independent set~$X$ into~$U$, it is immediate from the definition of~$F$ that if the perfect matching~$M$ exists in~$F$, then $\phi_4$ gives an embedding of~$H$ into~$G$. Hence it remains to prove that with positive probability~$\phi_3$ is such that~$F$ has a perfect matching. For this, we shall show that the properties guaranteed by the following lemma imply that Hall's condition is satisfied for~$F$.

\begin{lemma}[Stage~3]
\label{lem:stage3}
  Assuming Setup~\ref{setup} and that we perform Stages~1 to~3 of the randomised algorithm described above, with positive probability the embedding~$\phi_3$ has the following properties.
  \begin{enumerate}[label=\abc]
  \item\label{lem:stage3:a} 
  There are no $2\Delta^{2\eps}$-problematic vertices for~$\phi_3$.
  \item\label{lem:stage3:b}
  For all $v\in V$ we have
  \[\Big|\Big\{y\in Y\colon \phi_3\big(N_H(y)\big)\subset N_G(v)\Big\}\Big|\ge\frac{|Y|}{2\Delta^\eps}\,.\]
  \item\label{lem:stage3:c}
  For all $V'\subset V$ with $|V'|=\lceil\tfrac12|V|\Delta^{-3\eps/2-2}\rceil$ and $\delta_G(U,V')\ge\big(1-8 c(\log\Delta)\Delta^{\eps-1}\big)|V'|$ 
  we have
  \[\Big|\Big\{y\in Y\colon N_G\big(\phi_3(N_H(y))\big)\cap V'=\emptyset\Big\}\Big|\le \frac{|Y|}{2\Delta^{2\eps}}\,.\]
  \end{enumerate}
\end{lemma}

Property~\ref{lem:stage3:a} of this lemma asserts that each vertex $y\in Y$ is allowed to be embedded to many vertices in~$V$. This is complemented by~\ref{lem:stage3:b}, which states that each vertex $v\in V$ is a candidate host for many vertices in~$Y$. Together, these give us a minimum degree condition on the auxiliary bipartite graph $F$. Property~\ref{lem:stage3:c} states that for certain sets $V'\subset V$, only few vertices of~$Y$ cannot be embedded to~$V'$. Equivalently, in $F$ these sets $V'$ dominate most of $Y$. We will show that any large enough subset of $V$ contains one of these special $V'$ and hence this property tells us that any large enough subset of $V$ dominates most of $Y$ in $F$.

In order to prove Lemma~\ref{lem:stage3}, we shall rely on certain properties of the partial embeddings $\phi_1$ and~$\phi_2$, which are summarised in the following lemmas.

\begin{lemma}[Stage~1]
\label{lem:stage1}
  Assuming Setup~\ref{setup} and that we perform Stage~1 of the randomised algorithm described above, with positive probability the embedding~$\phi_1$ has the following properties.
  \begin{enumerate}[label=\abc]
  \item\label{lem:stage1:a} There are at most $2|X|\exp\big(-\Delta^{1-2\eps}\big)$ vertices which are $\Delta^\eps$-problematic for $\phi_1$, and 
  $|X'|\le 2\Delta\exp(-\Delta^{1-2\eps})|X|$.
  \item\label{lem:stage1:b}
  For all $v\in V$ we have
  \[\Big|\Big\{y\in Y\colon \phi_1\big(N_H(y)\setminus\tilde X\big)\subset N_G(v)\Big\}\Big|\ge\frac{|Y|}{\Delta^\eps}\,.\]
  \item\label{lem:stage1:c}
  For all $V'\subset V$ with $|V'|=\lceil\tfrac12|V|\Delta^{-3\eps/2-2}\rceil$ and $\delta_G(U,V')\ge\big(1-8c(\log\Delta)\Delta^{\eps-1}\big)|V'|$ 
  we have
  \[\Big|\Big\{y\in Y\colon N_G\big(\phi_1(N_H(y)\setminus\tilde X)\big)\cap V'=\emptyset\Big\}\Big|\le \frac{|Y|}{4\Delta^{2\eps}}\,.\]
  \item\label{lem:stage1:d}
  For all $v\in V$ we have
  \[\big|N_G(v)\cap(U\setminus\im\phi_1)\big|\ge\Big(1-3c\frac{\log\Delta}{\Delta}\Big)|U\setminus\im\phi_1|\,.\]
  \end{enumerate}
\end{lemma}

Properties \ref{lem:stage1:a}--\ref{lem:stage1:c} of this lemma are the Stage~1 analogues to Lemma \ref{lem:stage3}\ref{lem:stage3:a}--\ref{lem:stage3:c}, with the difference that we still may have some problematic vertices after Stage~1 (and hence~$X'$ may be non-empty). Property~\ref{lem:stage1:d} asserts that each vertex in~$V$ still has many neighbours in the part of~$U$ not used by $\phi_1$.
Observe that if $\Delta>\big(\log(2n)\big)^{1/(1-2\eps)}$
then~\ref{lem:stage1:a} states $|X'|=0$, so that $\im\phi_1=U$ (recall that in this case we chose the reservoir $\tilde{X}$ to be empty) and $\phi_1$ already satisfies the conclusion of Lemma~\ref{lem:stage3}.

We shall establish Lemma~\ref{lem:stage1} by using concentration inequalities for our random choice of~$\phi_1$. For analysing the vertices re-embedded in Stage~2, on the other hand, we shall rely on the Lov\'asz local lemma to obtain the following.

\begin{lemma}[Stage~2]
\label{lem:stage2}
  Assuming Setup~\ref{setup} and that we perform Stages~1 to~2 of the randomised algorithm described above, with positive probability the embedding~$\phi_2$ is such that there are no $\Delta^{2\eps}$-problematic vertices for $\phi_2$.
\end{lemma}

We defer the proofs of Lemma~\ref{lem:stage3}, Lemma~\ref{lem:stage1}, and Lemma~\ref{lem:stage2} to Section~\ref{sec:stages}. We conclude this section by showing that Lemma~\ref{lem:stage3} implies Theorem~\ref{thm:upper-bound}.

\begin{proof}[Proof of Theorem~\ref{thm:upper-bound}]
  Given an $n$-vertex graph $G$ with $\delta(G)\ge \big(1-c(\log\Delta)\Delta^{-1}\big)n$, and an $n$-vertex bipartite graph $H$ with maximum degree $\Delta$ and bipartition classes $X$ and $Y$. Without loss of generality we may assume $|X|\ge|Y|\ge|X|/\Delta$. To see that this is true, observe that if $|Y|<|X|/\Delta$, then necessarily $X$ contains vertices of degree zero, which we can move to $Y$ until the inequalities hold.

  If $\Delta\le \big(\log(2n)\big)^{1/(1-2\eps)}$, we greedily construct a reservoir $\tilde X\subset X$ of vertices of distance at least~$4$ and with $|\tilde X|=\lceil |X|/\Delta^{2+3\eps}\rceil$, which is possible since $\lceil n\Delta^{-2-3\eps}\rceil<\tfrac{|X|}{\Delta(\Delta-1)+1}$. If $\Delta> \big(\log(2n)\big)^{1/(1-2\eps)}$, we set $\tilde X=\emptyset$.
  Next, we apply
Lemma~\ref{lem:makebip} with $a=|X|$ and $\rho=\frac{c\log\Delta}{\Delta}$ 
to obtain a partition of $V(G)$ into two sets $U$ and $V$ with
$|U|=|X|$ and $|V|=|Y|$ such that
$\delta_G(U,V)\ge\big(1-2c\frac{\log\Delta}{\Delta}\big)|V|$
and $\delta_G(V,U)\ge\big(1-2c\frac{\log\Delta}{\Delta}\big)|U|$. We can apply this lemma since
\[\rho a \ge \rho(n-a) = \frac{c\log\Delta}{\Delta}|Y|
   \ge
    \frac{c\log\Delta}{\Delta(\Delta+1)}n
    \ge
   \frac{c\log(2\cdot e^{1000})}{2 (0.001n^{1/5})^2}n
   \ge 6\log n\,.
\]
 Hence, we are now exactly in 
Setup~\ref{setup}.
  
  We perform Stages~1 to~3 of the randomised algorithm described above. Suppose that the constructed embedding $\phi_3\colon X\to U$ satisfies the conclusions of Lemma~\ref{lem:stage3}.
  We need to show that for this choice of~$\phi_3$ the auxiliary bipartite graph~$F$ from Stage~4 of our algorithm contains a matching saturating~$V$. For this, we verify Hall's condition. Indeed, let $W\subset V$ be arbitrary. By Lemma~\ref{lem:stage3}\ref{lem:stage3:b} we have that $d_F(v) \geq \frac{|Y|}{2\Delta^\eps}=\frac{|V|}{2\Delta^\eps}$ and therefore,  if $|W|\le \frac{|V|}{2\Delta^\eps}$ we have that $|\bigcup_{v\in W}N_F(v)|\ge|W|$. If 
  $|W|\ge\Big(1-\frac{1}{2\Delta^{2\eps}}\Big)|V|$ then since by Lemma~\ref{lem:stage3}\ref{lem:stage3:a} there are no $2\Delta^{2\eps}$-problematic vertices for~$\phi_3$, each $y\in Y$ has degree $d_F(y)>\frac{|V|}{2\Delta^{2\eps}}$ and hence has a neighbour in $W$. We conclude that in this case $|\bigcup_{v\in W}N_F(v)|=|Y|\ge|W|$.

  Finally, suppose $\frac{|V|}{2\Delta^\eps}<|W|<\Big(1-\frac{1}{2\Delta^{2\eps}}\Big)|V|$. 
  By Setup~\ref{setup}, every vertex of $U$ has at most $2c\frac{\log\Delta}{\Delta}|V|$ non-neighbours in $V$, at most all of which are in $W$. Using the lower bound on $|W|$, this says every vertex of $U$ has at most $4 c(\log\Delta)\Delta^{\eps-1}|W|$ non-neighbours in $W$.
  We now choose a uniform random subset $W'$ of $W$ of size $\lceil\frac12|V|\Delta^{-3\eps/2-2}\rceil$. By Lemma~\ref{lem:hypgeo} with $\xi=1$ and the union bound, with probability at least 
  $1-n\cdot 2\exp\big(-\tfrac13\cdot 4c(\log\Delta)\Delta^{\eps-1}|W'|\big)$
  every vertex of $U$ has at most $8 c(\log\Delta)\Delta^{\eps-1}|W'|$ non-neighbours in $W'$. 
  This probability is positive since $|W'|\ge\lceil\frac12\frac{n}{\Delta+1}\Delta^{-3\eps/2-2}\rceil$
  and $\Delta\le 0.001n^{1/5}$, and therefore we can
  fix a $W'$ with this property. 
  
  By Lemma~\ref{lem:stage3}\ref{lem:stage3:c}, at most $|Y|/(2\Delta^{2\eps})$ vertices of $Y$ do not have an $F$-neighbour in $W'$, so \[\bigg|\bigcup_{v\in W}N_F(v)\bigg|\ge
  \bigg|\bigcup_{v\in W'}N_F(v)\bigg|\ge
  |Y|-\frac{|Y|}{2\Delta^{2\eps}}=\Big(1-\frac{1}{2\Delta^{2\eps}}\Big)|V|\ge|W|\,.\]

  This verifies Hall's condition for $F$. As described in Stage~4 of our algorithm, we now fix a perfect matching $M$ of $F$, and let $\phi_4\colon H\to G$ extend $\phi_3$, with $\phi_4(y)$ the matching partner of $y$ in $M$ for each $y\in Y$. By construction, $\phi_4$ is an embedding of $H$ into $G$.
\end{proof}

\section{The probability of having a small common neighbourhood}
\label{sec:lemma}

In this section we prove our main technical lemma (Lemma~\ref{lem:main}), which we shall need in the proofs of Lemma~\ref{lem:stage1} and Lemma~\ref{lem:stage2}. 
\begin{proof}[Proof of Lemma~\ref{lem:main}]
    Observe that it is enough to prove the statement for $t=\Delta$. Indeed, for an arbitrary $t\le\Delta$ we can generate a uniform random $\Delta$-set $T^{**}$ in $A'$ as the union of a uniform random $t$-set $T^*$ in $A'$ and a uniform random $(\Delta-t)$-set in $A'\setminus T^*$. The event that $T^*$ has a small common neighbourhood is contained in the event that $T^{**}$ has a small common neighbourhood, so it is enough to bound the probability of the latter event.
    
    Consider the following random process which generates a uniform random distribution over all sets $T\subset A'$ of size $\Delta$. We pick, one after another, distinct vertices $u_1,\dots,u_{\Delta}$ in $A'$ uniformly at random. Let $T=\{u_1,\dots,u_{\Delta}\}$, and set $S_0=B'$ and $S_i=S_{i-1}\cap N_{G'}(u_i)$ for each $i\in[\Delta]$. 

    We next define random variables $X_i$ for $i\in[\Delta]$ by setting
    \[X_i=\begin{cases}-\log \frac{|S_i|}{|S_{i-1}|}\quad&\text{if $|S_{i-1}|\ge \Delta^{-\eps}|B'|$}\,,\\ 0&\text{otherwise\,.}\end{cases}
    \]
    Let $\cE$ be the bad event that $|N_{G'}(T)|=|S_\Delta|\le\Delta^{-\eps}|B'|$.
    For any outcome in $\cE$, there is some smallest $t$ such that $|S_t|\le\Delta^{-\eps}|B'|$. 
    Note that $\sum_{i=1}^tX_i=\log|S_0|-\log|S_t|$.    
    It follows that in this event we have
    \[\sum_{i=1}^\Delta X_i\geq\sum_{i=1}^tX_i=\log|S_0|-\log|S_t|\ge\log|B'|-\log\big(\Delta^{-\eps}|B'|\big) =\eps\log\Delta\,.\]
    Therefore, to prove the lemma it suffices to show that
    \[\Prob\bigg[\sum_{i=1}^\Delta X_i\ge \eps\log\Delta\bigg]\le\exp\big(-\Delta^{1-2\eps}\big)\,.\]

    To establish this bound, we shall use Theorem~\ref{thm:freedman} with
    \[R=\frac18\eps\Delta^{2\eps-1}\log\Delta\quad\text{and}\quad
      \mu=\frac12\eps\log\Delta\,\]
    on the probability space $\Omega$ that contains all possible sequences $u_1,\dots,u_{\Delta}$ of picked vertices and with the filtration given by letting $\cF_i$ be the partition of $\Omega$ obtained from all possible choices of $u_1,\dots,u_i$. Then, clearly $X_i$ is $\cF_i$-measurable.
    
    We need to estimate the range and conditional expectation of the variables $X_i$. For this, we may assume $|S_{i-1}|\ge\Delta^{-\eps}|B'|$, since otherwise $X_i=0$ and our estimates will hold trivially. 
    Observe first that, whatever $u_i\in A'$ we pick, at most all of its $\frac1{16}\eps\Delta^{\eps-1}(\log\Delta)|B'|$ non-neighbours are in $S_{i-1}$, so we have
    \begin{equation}\label{eq:Si}
    \frac{|S_{i-1}\setminus N_{G'}(u_i)|}{|S_{i-1}|}
    \le
    \frac{\eps\Delta^{\eps-1}(\log\Delta)|B'|}{16|S_{i-1}|} 
    \le\frac{\eps\Delta^{\eps-1}(\log\Delta)|B'|}{16\Delta^{-\eps}|B'|} =\frac1{16}\eps\Delta^{2\eps-1}\log\Delta\le\frac12\,,
    \end{equation}
    where the last inequality uses that  $\eps\in(0,\tfrac13]$ and $\Delta\ge 2$.
    Hence, using our logarithm 
    estimate~\eqref{eq:log} with
$x=\frac{|S_{i-1}\setminus N_{G'}(u_i)|}{|S_{i-1}|}$, we conclude
    \begin{equation}\label{eq:Xi}
    X_i=-\log\Big(\frac{|N_{G'}(u_i)\cap S_{i-1}|}{|S_{i-1}|}\Big)=-\log\Big(1-\frac{|S_{i-1}\setminus N_{G'}(u_i)|}{|S_{i-1}|}\Big)
    \le 2\frac{|S_{i-1}\setminus N_{G'}(u_i)|}{|S_{i-1}|}
    \,.
    \end{equation}
    For the range of $X_i$ we immediately infer
    \begin{equation*}
    0\le X_i
    \leByRef{eq:Xi} 2\frac{|S_{i-1}\setminus N_{G'}(u_i)|}{|S_{i-1}|}
    \leByRef{eq:Si} \frac18\eps\Delta^{2\eps-1}\log\Delta=R\,.
    \end{equation*}
    For the conditional expectation, assume that $u_1,\dots,u_{i-1}$ and therefore $S_{i-1}$ are already fixed and thus determine the part of $\cF_{i-1}$ we are in. 
    We want to bound
    \[\Exp[X_i|\cF_{i-1}]\leByRef{eq:Xi} \Exp\bigg[2\frac{|S_{i-1}\setminus N_{G'}(u_i)|}{|S_{i-1}|}\bigg|\cF_{i-1}\bigg]\,.\]
    For this, note that each vertex of $S_{i-1}$ has at most \[\frac{1}{16}\eps\Delta^{-1}(\log\Delta)|A'|\le \frac18\eps\Delta^{-1}(\log\Delta)\big|A'\setminus\{u_1,\dots,u_{i-1}\}\big|\]
    non-neighbours in $A'\setminus\{u_1,\dots,u_{i-1}\}$, 
    where the inequality uses $|A'|\ge 2\Delta$.
    So the number of non-edges between $A'\setminus\{u_1,\dots,u_{i-1}\}$ and $S_{i-1}$ is at most $\tfrac18\eps\Delta^{-1}(\log\Delta)|A'\setminus\{u_1,\dots,u_{i-1}\}||S_{i-1}|$. 
    Since~$u_i$ is picked uniformly from $A'\setminus\{u_1,\dots,u_{i-1}\}$, we thus have
    \[\Exp\big[|S_{i-1}\setminus N_{G'}(u_i)|\,\big|\cF_{i-1}\big]\le \frac18\eps\Delta^{-1}(\log\Delta)|S_{i-1}|
    \]
    and hence
    \[\Exp[X_i|\cF_{i-1}]\le \Exp\bigg[2\frac{|S_{i-1}\setminus N_{G'}(u_i)|}{|S_{i-1}|}\bigg|\cF_{i-1}\bigg]\le \frac14\eps\Delta^{-1}\log\Delta\le\frac{\mu}{\Delta}\,. \]
    We conclude from Theorem~\ref{thm:freedman} that
    \[\Prob\Big[\sum_{i=1}^\Delta X_i\ge \eps\log\Delta\Big]
    \le\exp\Big(-\frac{\mu}{4R}\Big)
    =\exp\Big(-\frac{\tfrac12\eps\log\Delta}{4\cdot \tfrac18\eps\Delta^{2\eps-1}\log\Delta}\Big)
    =\exp\big(-\Delta^{1-2\eps}\big)\,,
    \]
    as desired.
\end{proof}

\section{Analysing the stages of our embedding algorithm}
\label{sec:stages}

In this section we first prove Lemma~\ref{lem:stage1}, then Lemma~\ref{lem:stage2}, and then Lemma~\ref{lem:stage3}. 
Lemma~\ref{lem:stage1} handles Stage~1, in which we pick a uniform random embedding of $X\setminus\tilde{X}$ into $U$. The idea of its proof is as follows. We use the Main Lemma and Markov's inequality to argue that there are likely to be few problematic vertices as required for~\ref{lem:stage1:a}.
For~\ref{lem:stage1:b}, we fix $v\in V$ and think of the random embedding as being generated by embedding sets $N_H(y)$, for $y\in Y$, one after another uniformly at random to the still available vertices of $U$. We then estimate the probability that $N_H(y)$ will be embedded to $N_G(v)$, and show it is likely that about the expected number of vertices in $Y$ have all their neighbours embedded to $N_G(v)$. To avoid dependencies, we actually split $Y$ into sets of vertices at pairwise distance more than $2$ and argue that concentration holds for each set of vertices individually.

For~\ref{lem:stage1:c}, the idea is to use the Main Lemma to argue that for a given $V'$ it is very unlikely that the following happens: for a large number of vertices $y\in Y$, we consistently happen to embed $N_H(y)$ to a set with few common neighbours in $V'$. As with~\ref{lem:stage1:b}, we want to avoid dependency and hence what we actually do, given a large subset of $Y$, is select greedily a large subset whose pairwise distance is more than $2$. We then take a union bound over all $V'$ of the size specified in~\ref{lem:stage1:c}. We remark that we would not be able to take a union bound over all much larger sets, which is the reason why we need the slightly strange-looking statement~\ref{lem:stage1:b}.
Finally,~\ref{lem:stage1:d} is an easy application of concentration of the hypergeometric distribution.

\begin{proof}[Proof of Lemma~\ref{lem:stage1}]
 Given $H$ with partition $X\dcup Y$, a set $\tilde X\subset X$, and $G$ with its partition $V(G)=U\dcup V$ as in Setup~\ref{setup}, recall that $\phi_1$ is a uniform random map from $X\setminus\tilde X$ to $U$. In particular, each subset of $X\setminus\tilde X$ is embedded to a uniform random set of the same size in $U$. We shall show that~\ref{lem:stage1:a} holds with probability at least $1/2$, and that the probability of~\ref{lem:stage1:b} or~\ref{lem:stage1:c} or~\ref{lem:stage1:d} not holding tends to zero as~$n$ tends to infinity.

 We first establish that~\ref{lem:stage1:a} holds with probability at least $1/2$. Given $y\in Y$, let $S_y=N_H(y)\setminus\tilde X$. Then $\phi_1(S_y)$ is a uniform random subset of $U$ of size $|S_y|\le\Delta$. We apply Lemma~\ref{lem:main} with input $\eps$ and $\Delta$ to $(U,V)$, which we can do since $\eps=1/3$ implies $2c=2/500<\eps/16$ and thus Setup~\ref{setup} gives
 \begin{align*}
    \delta_G(U,V)&\ge\Big(1-2c\frac{\log\Delta}{\Delta}\Big)|V|
    \ge
\Big(1-\frac{\eps\log\Delta}{16\Delta}\Delta^{\eps}\Big)|V|   
    \,, \\
\intertext{and}
 \delta_G(V,U)&\ge\Big(1-2c\frac{\log\Delta}{\Delta}\Big)|U|\ge \Big(1-\frac{\eps\log\Delta}{16\Delta}\Big)|U|\,. 
 \end{align*}
 This lemma tells us that the probability that $\phi_1(S_y)$ has fewer than $\Delta^{-\eps}|V|$ common neighbours in~$V$ is at most $\exp\big(-\Delta^{1-2\eps}\big)$. Summing over $y\in Y$, the expected number of $\Delta^{\eps}$-problematic vertices for $\phi_1$ is at most $|Y|\exp\big(-\Delta^{1-2\eps}\big)$. In particular, by Markov's inequality, with probability at least $\tfrac12$ we have $|Y_1|\le 2|Y|\exp\big(-\Delta^{1-2\eps}\big)$, and when this good event occurs we have $|X'|\le \Delta|Y_1|\le \Delta\cdot 2|X|\exp\big(-\Delta^{1-2\eps}\big)$, using that $|X|\ge|Y|$.

 We next show that~\ref{lem:stage1:b} fails with probability tending to~$0$. Drawing a graph $F$ on $Y$ in which $yy'$ is an edge if $y$ and $y'$ have a common neighbour in $H$, we see $\Delta(F)\le\Delta(\Delta-1)$ and hence by the Hajnal-Szemer\'edi theorem there is a proper vertex colouring of $F$ with $\Delta^2$ colours in which parts differ in size by at most $1$. In particular all parts have size between $\tfrac12|Y|\Delta^{-2}$ and $2|Y|\Delta^{-2}$, and any pair of vertices in such a part have no common neighbours in~$H$.
 Fix $v$ and a part $Y^*$ of this colouring. We now estimate the probability of the event $\cB(v,Y^*)$ that fewer than $\tfrac12 |Y^*|\Delta^{-8c}$ vertices $y\in Y^*$ satisfy $\phi_1(N_H(y)\setminus\tilde X)\subseteq N_G(v)$. For this, enumerate $Y^*=\{y_1,\dots,y_t\}$ arbitrarily, and observe that $\phi_1$ distributes the sets $N_H(y_i)\setminus\tilde X$ as the following random process since the $N_H(y_i)\setminus\tilde X$ are pairwise disjoint: for each time $1\le i\le t$ in turn, choose a uniform random subset of $U\setminus Z_{i-1}$ of size $|N_H(y_i)\setminus\tilde X|$, where $Z_{i-1}$ is the set of vertices chosen at time up to $i-1$.

 Since $|Z_{i-1}|\le\Delta|Y^*|\le2\Delta^{-1}|Y|\le 2\Delta^{-1}|X|=2\Delta^{-1}|U|$ we have
$|U\setminus Z_{i-1}|\ge(1-2\Delta^{-1})|U|$ and thus
$|Z_{i-1}|\le 2\Delta^{-1}|U\setminus Z_{i-1}|(1-2\Delta^{-1})^{-1}
\le c(\log\Delta)\Delta^{-1}|U\setminus Z_{i-1}|$, where we use that $\Delta\ge 2\cdot e^{1000}$ by Setup~\ref{setup}.
This implies
 \begin{equation*}\begin{split}
 \deg_G(v,U\setminus Z_{i-1})
 &\ge\left(1-2c\frac{\log\Delta}{\Delta}\right)|U|-|Z_{i-1}|
 \ge\left(1-2c\frac{\log\Delta}{\Delta}\right)|U\setminus Z_{i-1}|-c\frac{\log\Delta}{\Delta}|U\setminus Z_{i-1}|  \\
 &=\left(1-3c\frac{\log\Delta}{\Delta}\right)|U\setminus Z_{i-1}|\,.
 \end{split}\end{equation*}
 It follows that the probability that, conditioning on $Z_{i-1}$, a uniform random subset of $U\setminus Z_{i-1}$ of size $|N_H(y_i)\setminus\tilde X|$ is contained in $N_G(v)$, is at least
 \begin{equation*}\begin{split}
\frac{\binom{(1-3c(\log\Delta)\Delta^{-1})|U\setminus Z_{i-1}|}{\Delta}}{\binom{|U\setminus Z_{i-1}|}{\Delta}}&
\ge\left(\frac{(1-3c(\log\Delta)\Delta^{-1})|U\setminus Z_{i-1}|-\Delta}{|U\setminus Z_{i-1}|}\right)^{\Delta}
\ge \big(1-4c(\log\Delta)\Delta^{-1}\big)^{\Delta}\\
&\ge \exp(-8c\log\Delta)=\Delta^{-8c}\,,
 \end{split}\end{equation*}
 where the second inequality uses that $\Delta\le 0.001n^{1/5}$
 %$\Delta\le 0.01n^{1/2}$
 implies $\Delta<c(\log\Delta)\Delta^{-1}\cdot|U\setminus Z_{i-1}|$, and the third inequality uses the logarithm estimate~\eqref{eq:log} with
 $x=4c(\log\Delta)\Delta^{-1}\le\frac12$.

 This shows that the $t=|Y^*|\ge\tfrac12|Y|\Delta^{-2}$ indicators of events $\cB(v,Y^*,i)$ that $\phi_1(N_H(y_i)\setminus\tilde X)\subset N_G(v)$ stochastically dominate a collection of $t$ independent Bernoulli random variables, each having success probability $\Delta^{-8c}$. By Chernoff's inequality (Lemma~\ref{lem:hypgeo}, with $\xi=\tfrac12$), we get that
 \[
 \Prob\big(\cB(v,Y^*)\big)=
 \Prob\Big(\big|\{i\colon \cB(v,Y^*,i) \text{ occurs}\}\big|\le\frac12 |Y^*|\Delta^{-8c}\Big)
 \le\exp\Big(-\frac{1}{12}\cdot \frac12|Y|\Delta^{-2}\Delta^{-8c}\Big)\,.\] 
 Taking the union bound over the at most $n$ choices of $v$ and the $\Delta^2$ choices of $Y^*$, the probability that there exists $(v,Y^*)$ such that $\cB(v,Y^*)$ occurs tends to zero, since $\Delta\le 0.001n^{1/5}$.
 Now, if no $\cB(v,Y^*)$ occurs, then for any given $v$, summing over the choices of parts~$Y^*$ we obtain that at least $\tfrac12\Delta^{-8c}|Y|\ge\Delta^{-\eps}|Y|$ vertices $y\in Y$ satisfy $\phi_1(N_H(y)\setminus\tilde X)\subseteq N_G(v)$, as required for~\ref{lem:stage1:b}.

 Next we prove that~\ref{lem:stage1:c} fails with probability tending to~$0$. Fix $V'\subset V$ with $|V'|=\lceil\tfrac12|Y|\Delta^{-3\eps/2-2}\rceil$ and $\delta_G(U,V')\ge\big(1-8c(\log\Delta)\Delta^{\eps-1}\big)|V'|$. Suppose that $V'$ witnesses the good event of~\ref{lem:stage1:c} failing. This means there is a set $Y'$ of size greater than $\tfrac14|Y|\Delta^{-2\eps}$ such that $N_G\big(\phi_1(N_H(y)\setminus\tilde X)\big)\cap V'=\emptyset$ for each $y\in Y'$. Choose a maximal subset $Y'_0$ of $Y'$ such that the sets $N_H(y)$ for $y\in Y'_0$ are pairwise disjoint. We have $|Y'_0|\big(\Delta(\Delta-1)+1\big)\ge|Y'|$ and so $|Y'_0|\ge\tfrac14|Y|\Delta^{-2\eps-2}$. Fix an arbitrary subset $Y''$ of~$Y'_0$ of size $t:=\lceil\tfrac14|Y|\Delta^{-2\eps-2}\rceil$.

 We now estimate the probability of the event $\cE(V',Y'')$ that $N_G\big(\phi_1(N_H(y)\setminus\tilde X)\big)\cap V'=\emptyset$ for each $y\in Y''$. For this, enumerate $Y''=\{y_1,\dots,y_t\}$ arbitrarily. As before, $\phi_1$ distributes the sets $N_H(y_i)\setminus\tilde{X}$ as the following random process: for each time $1\le i\le t$ in turn, choose a uniform random subset of $U\setminus Z_{i-1}$ of size $|N_H(y_i)\setminus\tilde{X}|$, where $Z_{i-1}$ is the set of vertices chosen at time up to $i-1$.
 Consider the bipartite graph $G[U\setminus Z_{i-1},V']$. We want to apply our main lemma, Lemma~\ref{lem:main}, to $G[U\setminus Z_{i-1},V']$. Hence, we next check the degree conditions of this lemma. 
 By choice of $V'$, every vertex of $U\setminus Z_{i-1}$ has at least $\big(1-8c(\log\Delta)\Delta^{\eps-1}\big)|V'|\ge\big(1-\frac1{16}\eps(\log\Delta)\Delta^{\eps-1}\big)|V'|$ neighbours in $V'$. Since $|Z_{i-1}|\le \Delta|Y''|\le\tfrac12|Y|\le\tfrac12|U|$, every vertex of $V'$ has at least $\big(1-2c(\log\Delta)\Delta^{-1}\big)|U|-|Z_{i-1}|\ge\big(1-4c(\log\Delta)\Delta^{-1}\big)|U\setminus Z_{i-1}|\ge \big(1-\frac1{16}\eps(\log\Delta)\Delta^{-1}\big)|U\setminus Z_{i-1}|$ neighbours in $U\setminus Z_{i-1}$. Hence, we can apply Lemma~\ref{lem:main}, which implies that, conditioned on the choice of $Z_{i-1}$, our random process chooses 
 a random subset of $U\setminus Z_{i-1}$ of size $|N_H(y_i)\setminus\tilde{X}|$ which has no
 common neighbours in $V'$ with probability at most
 $\exp\big(-\Delta^{1-2\eps}\big)$, where we use that $\Delta^{-\eps}|V'|\ge 0$.
 
 Multiplying these conditional probabilities, the probability of the event $\cE(V',Y'')$ is at most $\exp\big(-\Delta^{1-2\eps}t\big)$. We now take the union bound over the choice of $Y''$ and $V'$, which are of sizes $t$ and $\lceil\tfrac12|Y|\Delta^{-3\eps/2-2}\rceil\le 4\Delta^{\eps/2}t$ respectively, obtaining that the probability that there are $V'$ and $Y''$ such that $\cE(V',Y'')$ holds is at most
 \begin{multline*}
 \binom{n}{t}\binom{n}{4\Delta^{\eps/2}t}\exp\big(-\Delta^{1-2\eps}t\big)
 \le\Big(\frac{en}{t}\Big)^t\Big(\frac{en}{t}\Big)^{4\Delta^{\eps/2}t}\exp\big(-\Delta^{1-2\eps}t\big) \\
 \le\exp\big(8t\Delta^{\eps/2}\log(en/t)-\Delta^{1-2\eps}t\big)
 =\exp\Big(8t\Delta^{\eps/2}\big(\log(en/t)-\tfrac18\Delta^{1-5\eps/2}\big)\Big)
 \,,
 \end{multline*}
 and hence~\ref{lem:stage1:c} fails with at most that probability.
 Using that $|Y|\ge n/(\Delta+1)$ and hence $t=\lceil\tfrac14|Y|\Delta^{-2\eps-2}\rceil\ge\tfrac18n\Delta^{-2\eps-3}$, we see that this probability tends to~$0$ because
 \[\log(en/t)\le\log\Big(\frac{en}{\tfrac18n\Delta^{-2\eps-3}}\Big)=\log(8e\Delta^{3+2\eps})<\frac18\Delta^{1-5\eps/2}\]
 for $\eps=1/3$ and $\Delta\ge 2\cdot e^{1000}$.

 Finally we prove~\ref{lem:stage1:d}. Observe that $U\setminus \im\phi_1$ is distributed as a uniform random subset of $U$ of size $|\tilde X|$. If $\tilde X=\emptyset$, then~\ref{lem:stage1:d} is trivially true, so we suppose $|\tilde X|=\lceil|X|\Delta^{-2-3\eps}\rceil$ and $\Delta\le(\log(2n))^{1/(1-2\eps)}$ as in Setup~\ref{setup}. For a fixed $v\in V$, pick a superset $S$ of $U\setminus N_G(v)$ of size $2c(\log\Delta)\Delta^{-1}|U|$. Then $|S\cap (U\setminus \im\phi_1)|$ is hypergeometrically distributed with parameters $(|U|,2c(\log\Delta)\Delta^{-1}|U|,|\tilde X|)$, hence has expectation $2c(\log\Delta)\Delta^{-1}|\tilde X|$. So by Lemma~\ref{lem:hypgeo} with $\xi=\frac14$, the probability that $|S\cap (U\setminus \im\phi_1)|$ exceeds $\frac52c(\log\Delta)\Delta^{-1}|\tilde X|=\frac52c(\log\Delta)\Delta^{-1}|U\setminus\im\phi_1|$ is at most $2\exp\big(-\tfrac1{48}\cdot 2c(\log\Delta)\Delta^{-1}|\tilde X|\big)$. Taking the union bound over~$v$, the probability that there exists a vertex~$v$ with
 $|N_G(v) \cap (U\setminus \im\phi_1)|\le(1-3c(\log\Delta)\Delta^{-1})|U\setminus\im\phi_1|$
 tends to zero since $\Delta\le 0.001n^{1/5}$, as required.
\end{proof}

We next prove Lemma~\ref{lem:stage2} using the Lov\'asz local lemma. Recall that if $\Delta$ is large, there are no problematic vertices and the lemma holds trivially, so we can assume $\Delta$ is small and the reservoir satisfies $|\tilde{X}|\ge|X|\Delta^{-2-3\eps}.$ We simply re-embed uniformly at randomly the vertices $X'$ adjacent to problematic vertices in $Y$. Since these vertices were embedded uniformly at random in Stage~1, the reader can reasonably ask what we gain. The answer is that in Stage~1, we needed the failure probability to be reasonably small in order to make a union bound work with the failure probabilities for the other properties guaranteed by Lemma~\ref{lem:stage1}, whereas in this lemma we are only interested in problematic vertices and can afford to have a failure probability very close to, but strictly smaller than, $1$. It is in this step where we need the reservoir $\tilde{X}$: even conditioning on the embedding of all but the last few vertices of $X\setminus\tilde{X}$, because the reservoir is not embedded there are still many choices for these last few vertices. We then use this observation together with our main lemma to argue that it is unlikely that any given vertex that was problematic for $\phi_1$ remains problematic for $\phi_2$.

\begin{proof}[Proof of Lemma~\ref{lem:stage2}]
In the following we condition on the event that the embedding $\phi_1$ produced by Stage~1 satisfies the conclusion of Lemma~\ref{lem:stage1}, which happens with positive probability.
Recall that we can assume that $\Delta\le\big(\log(2n)\big)^{1/(1-2\eps)}$, as otherwise $\phi_2=\phi_1$ and Lemma~\ref{lem:stage1}\ref{lem:stage1:a} implies that that there are no $\Delta^{\eps}$-problematic vertices for~$\phi_2$. Also observe that every $y\in Y\setminus Y'$ is not $\Delta^\eps$-problematic, and hence not $\Delta^{2\eps}$-problematic, for $\phi_2$, since $y\notin Y_1$ and $\phi_2$ agrees with $\phi_1$ on $N_H(y)\setminus\tilde X$.

Recall that~$X'$ is the set of neighbours of $\Delta^{\eps}$-problematic vertices, and let $Y'$ be the vertices in $Y$ with a neighbour in $X'$. Let $\phi_2$ be obtained as described in Stage 2, i.e. by keeping vertices of $X_1=X\setminus(\tilde X\cup X')$ embedded as in $\phi_1$ and by choosing a uniform random injective map from $X'$ to $U\setminus\phi_1(X_1)$.
For each $y\in Y'$, let $\cE_y$ be the event that $y$ is $\Delta^{2\eps}$-problematic for $\phi_2$, i.e. that $\big|N_G\big(\phi_2(N_H(y)\setminus\tilde X)\big)\cap V\big|\le |V|\Delta^{-2\eps}$.
We would like to use Lemma~\ref{lem:LLLL} to show that the probability that no event $\cE_y$ with $y\in Y'$ occurs is positive. 
Let $p=\exp(-\Delta^{1-2\eps})$. We claim that the graph $D$ on vertex set $\{\cE_y\colon y\in Y'\}$
with edges $\cE_y\cE_{y'}$ for each pair $y,y'\in Y'$ of distinct vertices such that $N_H(y)\cap N_H(y')\neq\emptyset$
is a $p$-dependency graph. Observe that this claim, by Lemma~\ref{lem:LLLL}, completes the proof, since $D$ has maximum degree at most $\Delta^2$ and we have $4\Delta^2p<1$ since $\eps=1/3$ and $\Delta\ge 2e^{1000}$.

To verify the claim, consider $y\in Y'$, and let $S=N_H(y)\cap X_1$. Given a set of events not in $N_D(\cE_y)$, let $Y''$ be the corresponding set of vertices of $Y'$, and let $X''$ be the vertices in $X'$ with a neighbour in $Y''$. By definition of $D$, no vertex of $X''$ is adjacent to $y$. We need to show that $\Prob\big[\cE_y|\bigcap_{y''\in Y''}\bar\cE_{y''}\big]\le p$. Observe that the event $\bigcap_{y''\in Y''}\bar\cE_{y''}$ only depends on the embedding of~$X''$ by~$\phi_2$. Hence, it is sufficient to prove $\Prob[\cE_y|\phi_2(X'')]\le p$ for all choices of $\phi_2(X'')$. Observe further that, conditioning on $\phi_2(X'')$, the distribution of $N_H(y)\setminus(\tilde X\cup X_1)$ is a uniform random subset of $U\setminus(\phi_1(X_1)\cup\phi_2(X''))$ of size $|N_H(y)\setminus(\tilde X\cup X_1)|\le\Delta$ since~$y$ has no neighbours in~$X''$.
We want to apply Lemma~\ref{lem:main} with $A'=U\setminus\big(\phi_1(X_1)\cup\phi_2(X'')\big)$ and $B'=N_G\big(\phi_1(S)\big)\cap V$. 
Since we already assumed that $\Delta\le\big(\log(2n)\big)^{1/(1-2\eps)}$, therefore $|A'| \geq |\tilde X| >2\Delta$. Thus
we only need to check the degree conditions for this lemma.

We first verify the degree condition for vertices $u\in A'$. If $y\in Y_1$, then $S=\emptyset$ by the definition of $X_1$, and $|B'|=|V|$. Therefore, $u$ has at most $2c(\log\Delta)\Delta^{-1}|B'|$ non-neighbours in~$B'$ in this case and $2c\le\eps/16$, as required.
If $y\not\in Y_1$, then $y$ is not $\Delta^\eps$-problematic for $\phi_1$ and hence $|B'|=\big|N_G\big(\phi_1(S)\big)\cap V\big|\ge|V|\Delta^{-\eps}$. Therefore, in this case $u$ has at most 
$2c(\log\Delta)\Delta^{-1}|V|\le \frac{\eps}{16}(\log\Delta)\Delta^{\eps-1}|B'|$ non-neighbours in $B'$, giving the degree condition from~$A'$ to~$B'$.

For the degree condition from~$B'$ to~$A'$,
consider $v\in B'$. From Lemma~\ref{lem:stage1}\ref{lem:stage1:d} we know that~$v$ has at most $3c(\log\Delta)\Delta^{-1}|U\setminus\im\phi_1|$ non-neighbours in $U\setminus\im\phi_1$, and at most all of $\phi_1(X')$ are non-neighbours of $v$. By Lemma~\ref{lem:stage1}\ref{lem:stage1:a} we have $|X'|\le2\Delta\exp(-\Delta^{1-2\eps})|X|\le2 \Delta\exp(-\Delta^{1-2\eps})\Delta^{2+3\eps}|U\setminus\phi_1(X_1)|$ since $|\tilde X|\ge|X|/\Delta^{2+3\eps}$. Hence, the number of non-neighbours of~$v$ in $U\setminus\phi_1(X_1)$ is at most 
\begin{equation*}\begin{split}
3c\frac{\log\Delta}{\Delta}|U\setminus\im\phi_1|+|X'|
&\le 4c\frac{\log\Delta}{\Delta}|U\setminus\phi_1(X_1)|+
2\Delta^{3+3\eps}\exp(-\Delta^{1-2\eps})|U\setminus\phi_1(X_1)| \\
&\le
6c\frac{\log\Delta}{\Delta}\big|U\setminus\phi_1(X_1)\big|\,,
\end{split}
\end{equation*}
 where we use that $\Delta\ge 2\cdot e^{1000}$. At most all of these non-neighbours lie outside $\phi_2(X'')$, and since $|X''|\le|X'|\le 2\Delta\exp(-\Delta^{1-2\eps})|X|$, we see that $v$ has at most $8c(\log\Delta)\Delta^{-1}|A'|$ non-neighbours in $A'$. Since $8c\le\eps/16$, this verifies the degree condition from $B'$ to $A'$. 

Lemma~\ref{lem:main} applied to $A'=U\setminus\big(\phi_1(X_1)\cup\phi_2(X'')\big)$ and $B'=N_G\big(\phi_1(S)\big)\cap V$, for $T=\phi_2\big(N_H(y)\setminus(\tilde X\cup X_1)\big)$ tells us that, conditioned on $\phi_2(X'')$, the probability that $|N_G(T)\cap B'|$ is smaller than or equal to $\Delta^{-\eps}|B'|$ is at most
$\exp(-\Delta^{1-2\eps})=p$. However, if this event holds, then
\[\big|N_G\big(\phi_2(N_H(y)\setminus\tilde X)\big)\cap V\big|
= |N_G(T)\cap B'|
\le\Delta^{-\eps}|B'|\,.\]
Since $\Delta^{-2\eps}|V|\le \Delta^{-\eps}|B'|$, we conclude that, conditioning on $\phi_2(X'')$, 
the probability that $\cE_y$ holds
is at most $p$, as required.
\end{proof}

The proof of Lemma~\ref{lem:stage3} is now straightforward. We embed the reservoir vertices arbitrarily and just need to explain why the properties obtained in Lemmas~\ref{lem:stage1} and~\ref{lem:stage2} cannot degrade much. The basic reason for this is that we only embed at most one neighbour of any given $y\in Y$ in this stage, and the degree conditions we obtained provide bounds which are much larger than the number of non-neighbours of any given $u\in U$.

\begin{proof}[Proof of Lemma~\ref{lem:stage3}]
 If $\Delta>\big(\log(2n)\big)^{1/(1-2\eps)}$, then we chose $\tilde X=\emptyset$, and Lemma~\ref{lem:stage1} returns an embedding $\phi_1$ with the good properties of that lemma. Let $\phi_3=\phi_1$, and observe that $2\exp\big(-\Delta^{1-2\eps}\big)|X|<1$, so Lemma~\ref{lem:stage1}\ref{lem:stage1:a} states that $\phi_3$ has no $\Delta^\eps$-problematic vertices, and so also no $2\Delta^{2\eps}$-problematic vertices, establishing~\ref{lem:stage3:a}. The other two conclusions of Lemma~\ref{lem:stage3} are immediate from the corresponding statements of Lemma~\ref{lem:stage1}.

 We now suppose $\Delta\le\big(\log(2n)\big)^{1/(1-2\eps)}$.  Given embeddings $\phi_1$ and $\phi_2$ with the properties of Lemmas~\ref{lem:stage1} and~\ref{lem:stage2} respectively, we created $\phi_3$ by picking an arbitrary embedding of $\tilde X$ to $U\setminus\im\phi_2$.
 Any $y\in Y$ is not $\Delta^{2\eps}$-problematic for $\phi_2$ by Lemma~\ref{lem:stage2}. Since $\tilde X$ contains at most one neighbour of $y$, and since $\delta_G(U,V)\ge\big(1-2c(\log\Delta\big)\Delta^{-1})|V|$, this additional neighbour can decrease the common degree of previously embedded neighbours of~$y$ by at most $2c(\log\Delta)\Delta^{-1}|V|$. Thus we have
 \begin{align*}
    \big|N_G\big(\phi_3(N_H(y))\big)\cap V\big|&\ge\big|N_G\big(\phi_2(N_H(y)\setminus\tilde X)\big)\cap V\big|-2c(\log\Delta)\Delta^{-1}|V|\\
      &\ge |V|\Delta^{-2\eps}-2c(\log\Delta)\Delta^{-1}|V|
      \ge \tfrac12|V|\Delta^{-2\eps}\,,
 \end{align*}
 which shows that $y$ is not $2\Delta^{2\eps}$-problematic, establishing~\ref{lem:stage3:a}.

 Given $v\in V$, by Lemma~\ref{lem:stage1}\ref{lem:stage1:b} there are at least $|Y|\Delta^{-\eps}$ vertices $y\in Y$ such that $\phi_1\big(N_H(y)\setminus\tilde X\big)\subset N_G(v)$. By Lemma~\ref{lem:stage1}\ref{lem:stage1:a}, at most $2\Delta^2\exp(-\Delta^{1-2\eps})|X|$ of these vertices have a neighbour in $X'$, and a further at most $\Delta|\tilde X|$ have a neighbour in $\tilde X$, leaving at least
 \[|Y|\Delta^{-\eps}-2\Delta^2\exp(-\Delta^{1-2\eps})|X|-\Delta|\tilde X|\ge\frac{|Y|}{2\Delta^\eps}\]
 vertices $y$ which satisfy neither condition, where we used 
 $|X|\le(\Delta+1)|Y|$, $|\tilde X|\le2n/(\Delta^{2+3\eps})$ and
 $\Delta\ge 2\cdot e^{1000}$.  For such $y$ we have $\phi_3\big(N_H(y)\big)=\phi_1\big(N_H(y)\setminus\tilde X\big)\subset N_G(v)$, which proves~\ref{lem:stage3:b}.

 Finally, the proof of~\ref{lem:stage3:c} is analogous: given $V'$ satisfying the conditions of~\ref{lem:stage3:c}, by Lemma~\ref{lem:stage1}\ref{lem:stage1:c} there are at most $\tfrac14|Y|\Delta^{-2\eps}$ vertices $y$ of $Y$ such that $N_G\big(\phi_1(N_H(y)\setminus\tilde X)\big)\cap V'=\emptyset$, and at most a further $2\Delta^2\exp(-\Delta^{1-2\eps})|X|+\Delta|\tilde X|$ are added to this as $\phi_1$ is changed to $\phi_3$, which gives~\ref{lem:stage3:c}.
\end{proof}

\section{Concluding remarks}
\label{sec:concl}

We have shown that, for bipartite spanning graphs of maximum degree
$\Delta$, the Bollob\'as--Eldridge--Catlin minimum-degree threshold can
be improved by a logarithmic factor, and that this improvement is best
possible up to the value of the constant. It would be interesting to
understand how far beyond the bipartite setting this phenomenon extends.

The first natural question is whether the same logarithmic gain holds
for graphs of bounded chromatic number. More precisely, for every fixed
integer $r\ge2$, does there exist a constant $c_r>0$ such that every
$n$-vertex graph $G$ with
$
\delta(G)\ge
\Big(1-c_r\frac{\log\Delta}{\Delta}\Big)n
$
contains every $n$-vertex graph $H$ with maximum degree at most
$\Delta$ and chromatic number at most $r$, provided $\Delta$ is not too
large compared to $n$?

A still more speculative question is whether the bipartite assumption in our
main theorem can be replaced by the much weaker assumption that $H$ is
triangle-free, while still retaining a logarithmic improvement over the
Bollob\'as--Eldridge--Catlin threshold.

\bibliographystyle{amsplain_yk}
\bibliography{main}
\end{document}